\documentclass[a4paper,11pt,reqno]{amsart}
\usepackage{amsmath,amssymb,amsthm}
\usepackage{dsfont}
\usepackage[utf8]{inputenc}
\usepackage[T1]{fontenc}

\newcommand{\calB}{\mathcal{B}}
\newcommand{\calD}{\mathcal{D}}
\newcommand{\calE}{\mathcal{E}}

\newcommand{\calM}{\mathcal{M}}

\newcommand{\calS}{\mathcal{S}}
\newcommand{\calL}{\mathcal{L}}
\newcommand{\Ell}{\calL}

\newcommand{\NN}{\mathds{N}}
\newcommand{\RR}{\mathds{R}}

\newcommand{\pex}{\mathds{P}^x}
\newcommand{\eex}{\mathds{E}^x}

\DeclareMathOperator{\supp}{supp}
\DeclareMathOperator{\capp}{cap}

\newcommand{\One}{\mathds{1}}
\newcommand{\rmd}{\mathrm{d}} 
\newcommand{\with}{:}

\newcommand{\set}[2]{\bigl\{#1\bigm|#2\bigr\}}

\renewcommand\le{\leqslant}
\renewcommand\ge{\geqslant}

\renewcommand{\c}{{\rm c}}

\newtheorem{theorem}{Theorem}[section]
\newtheorem{corollary}[theorem]{Corollary}
\newtheorem{lemma}[theorem]{Lemma}
\newtheorem{proposition}[theorem]{Proposition}

\theoremstyle{definition}
\newtheorem{definition}[theorem]{Definition}

\theoremstyle{remark}
\newtheorem{remark}[theorem]{Remark}

\def\clap#1{\hbox to 0pt {\hss#1\hss}}

\newcommand{\Hmm}[1]{\leavevmode{\marginpar{\tiny%
$\hbox to 0mm{\hspace*{-0.5mm}$\leftarrow$\hss}%
\vcenter{\vrule depth 0.1mm height 0.1mm width \the\marginparwidth}%
\hbox to 0mm{\hss$\rightarrow$\hspace*{-0.5mm}}$\\\relax\raggedright
#1}}}

\begin{document}

\title[A Persson theorem for general Dirichlet forms]
{On the decomposition principle and a Persson type theorem for
general regular  Dirichlet forms}

\author[]{Daniel Lenz$^1$}
\author[]{Peter Stollmann$^2$}

\address{ $^1$ Mathematisches Institut, Friedrich Schiller Universit\"at Jena,
  D-03477 Jena, Germany, daniel.lenz@uni-jena.de,
  URL: http://www.analysis-lenz.uni-jena.de/ }
\address{$^2$ Fakult\"at f\"ur Mathematik, TU Chemnitz, 09107 Chemnitz,
Germany}

\date{22.5.2017}

\begin{abstract}
We present a decomposition principle for general regular Dirichlet
forms satisfying a spatial local compactness condition. We use
the decomposition principle to derive a Persson type theorem for the
corresponding Dirichlet forms.  In particular our setting covers
Laplace-Beltrami operators on Riemannian manifolds, and Dirichlet forms
associated to $\alpha$-stable
processes in Euclidean space.
\end{abstract}

\maketitle

\section*{Introduction}
Both decomposition principles and Persson's theorem deal with the
essential spectrum of  suitable Laplace type operators $\Ell$ acting on
a locally compact space $X$. With notation and underlying concepts
to be explained in further detail below the \textit{decomposition
principle} asserts
$$\sigma_{ess} (\Ell) = \sigma_{ess} (\Ell_{X\setminus K}).$$
Here,  $\sigma_{ess}$ denotes the essential spectrum and
$\Ell_{X\setminus K}$ denotes the restriction of $\Ell$ to the complement
of a compact set $K$.  On the other hand \textit{Persson's theorem}
states
$$\inf \sigma_{ess} (\Ell) = \lim_{K\to X} \inf \sigma (\Ell_{X\setminus
K}).$$ Here,  $\sigma$ denotes the spectrum and the limit is taken
along the net of compact subsets of $X$.

Both Persson's theorem and the decomposition principle are concrete
instances of the general philosophy that the essential spectrum of
Laplacians 'comes from infinity'. Moreover,  Persson's theorem can
be seen as a rather direct consequence of the decomposition
principle. Indeed, the equality asserted in Persson's theorem can be
thought of as two inequalities. Now, one of them  (viz '$\geq$')
immediately  follows from the decomposition principle and the other
inequality (viz '$\leq$')  can be shown rather directly and is, in
essence, well-known.

The decomposition principle has been a mainstake in spectral theory
for at least a hundred years. Indeed, already Weyl's celebrated 1910
paper \cite{Weyl} on absence of essential spectrum for Schr\"odinger
type operators of the form $- \Delta + V$ with $V(x)\to \infty$,
$x\to \infty$, on $L^2 (\RR^d)$ crucially relies on it. These ideas
have then been taken up later to even characterize the growth of $V$
yielding absence of essential spectrum in  \cite{Molchanov, MS} and
in \cite{LSW}. In fact, we also use some ideas and concepts from the
latter article.

In 1958 Persson then proved the result named after him for general
divergence type Schr\"odinger operators, \cite{Persson}, actually
without using a decomposition principle! His result  became an
important tool in corresponding considerations in mathematical
physics.

In a spirit very similar to Weyl's original  approach, Donnelly and
Li \cite{DL} stated and used a decomposition principle for
Riemannian manifolds in their famous  work  from 1979 showing
absence of essential spectrum for Laplace-Beltrami operators
whenever the sectional curvature of the underlying manifold tends to
$-\infty$.

In the middle of the  90ies the pioneering work of Sturm
\cite{Sturm-94,Sturm-95} made it clear that strongly local Dirichlet
forms provide a proper setting for and generalization  of spectral
geometry of  Laplacians on Euclidean space and on  manifolds. The
key ingredient of this approach is the so-called intrinsic metric.
To make the theory work one  has to assume that  this intrinsic metric
generates
the original topology. (As
shown in \cite{FLW} this assumption can be relaxed to continuity of
the intrinsic metric in various circumstances, see
\cite{Stollmann-10} as well). Under this compatibility assumption
 as well as some technical assumptions
 Grillo \cite{Grillo} proved in 1998 a general Persson type theorem for
strongly local Dirichlet forms. Quite remarkably, Grillo's work does
not contain a  decomposition principle either. We note in passing that
various other parts of general spectral theory could be shown for
strongly local Dirichlet forms satisfying the compatibility
condition in the last ten years \cite{BoutetS-03,BoutetLS-09,LenzSV-09}.

All the works mentioned so far deal with regular  strongly local
Dirichlet forms (corresponding to Markov processes with continuous
paths). Recent years have now seen an ever increasing interest in
non-local Dirichlet forms (corresponding to Markov processes with
jumps). Here, most prominent cases concern Laplacians on graphs and
$\alpha$-stable processes, see Section 4 below for a selection of
references on these topics

For  Laplacians on graphs with standard weights  Persson's theorem
can be found in  Keller \cite{Keller} (see as well the earlier work
of Fujiwara for related material dealing with the normalized
Laplacian \cite{Fujiwara}). This result was then generalized to
weighted locally finite graphs by Keller and Lenz \cite{KellerL-10}.
For these graphs validity of a decomposition principle is quite
straightforward and accordingly a Persson type theorem is rather
easy to derive. However, so far neither a decomposition principle
nor Persson's theorem are known for general graphs let alone more
complicated non-local situations such as $\alpha$-stable processes.
This is the starting point of the present note.

More specifically, we present a decomposition principle (Section
\ref{sec-decomposition}) and a Persson type theorem (Section
\ref{sec-Persson}) valid for all regular Dirichlet forms satisfying
a spatial local compactness condition.

This condition is then shown to be valid in a variety of situations
(Section \ref{examples}). These situations include the already known
case of Laplace-Beltrami operators on manifolds. More importantly,
they also include  rather general graphs as well as $\alpha$-stable
processes in Euclidean space. To the best of our knowledge these are
the first results giving a decomposition principle and a Persson
theorem for jump processes which are not coming from graphs.

All these considerations are given after a first introduction into
the topic of Dirichlet forms in Section \ref{Preliminaries}. We
finally discuss in Section  \ref{sec-perturbations} a generalization
of our main results to certain Schr\"odinger type operators. More
precisely, we will be able to deal with measure perturbations where
the negative part has to fulfill some Kato condition while the
positive part is quite arbitrary.

\section{Preliminaries on Dirichlet forms and selfadjoint operators}
\label{Preliminaries}
Our investigations are situated in the context of
  (regular)  Dirichlet forms on locally compact
separable spaces. In this section we introduce the necessary
background and notation  (see e.g.\ \cite{BouleauH-91, Davies-95,
FukushimaOT-94, MaR-92}). The basic relevance of regular Dirichlet
forms in our context  comes from the fact  that they are in
one-to-one correspondence to Markov processes with paths satisfying
a regularity condition known as \textit{c\`{a}dl\`{a}g}. Thus, Dirichlet
forms provide an analytic tool to work with Markov processes.

\medskip

Throughout we let $X$ be a locally compact, separable metric space
and  $m$ a positive Radon measure on X with $\supp m = X$.  We will
only consider real valued function on $X$.  (Of course, complex
valued functions could easily be considered as well after
complexifying the corresponding Hilbert spaces and forms.) By $C_c
(X)$ we denote the set of continuous functions on $X$ with compact
support and by $C_0 (X)$  the closure of $C_c (X)$ with respect to
the supremum norm.  The space $L^2 (X,m)$ is the space of classes of
measurable and (with respect to $m$) square integrable   real valued
functions. The norm on $L^2 (X,m)$ will be denoted by
$\|\cdot\|_{L^2 (X,m)}$ or, if the space is clear from the context,
just by $\|\cdot\|$.  The space $L^\infty (X,m)$ is the space of all
classes of essentially bounded functions.

\smallskip

The characteristic function of a set $M$  will be denoted by $\One_M$
i.e. $\One_M (x) = 1$ if $x\in M$ and $\One_M (x) = 0$ else.

\smallskip

A closed non-negative form on $L^2(X,m)$ consists of a dense
subspace $\calD \subset L^2(X,m)$ and a sesquilinear  map $\calE :
\calD \times \calD \to \RR$ with $\calE (f,f)\geq 0$ for all
$f\in\calD$ such that $\calD$ is complete with respect to the
\textit{energy norm} $\|\cdot\|_{\calE}$ defined by
\[  \|u\|_{\calE} : = \left( \calE(u,u) + \| u \|_{L^2(X,m)}^2\right)^{1/2}. \]
Whenever $\calE$ is a closed form there exists a unique selfadjoint
operator $\Ell$ on $L^2 (X,m)$ with
$$\langle \Ell u, v\rangle = \calE (u,v)$$
for all $u$ in the domain of $\Ell$ and all $v\in \calD$. This operator
$\Ell$ is non-negative.

A closed form is said to be a \emph{Dirichlet form} if for any $u
\in \calD$ and any normal contraction $C : \RR \to \RR$ we have also
\[ C \circ u \in \calD \mbox{ and } \calE(C \circ u) \leq \calE(u). \]
Here, $C : \RR \to \RR$ is called a \emph{normal contraction} if
$C(0)=0$ and $|C(\xi)-C(\zeta)|\le |\xi -\zeta|$ for any
$\xi,\zeta\in\RR$. Typical examples of normal contractions are
modulus and positive and negative part of a number.  A Dirichlet
form is called \emph{regular} if $\calD \cap C_c(X)$ is dense both
in $(\calD, \| \cdot \|_{\|\cdot\|_{\calE}})$ and $(C_c(X), \| \cdot
\|_{\infty })$.

By  the fundamental connection to probability theory, see
Fukushima's classic \cite{Fukushima-80}, to every regular Dirichlet
form there exists an associated Markov process $(\Omega, (\pex)_{x\in
X}, (X_t)_{t\ge 0})$ with state space $X\cup \{\infty\}$  related to the form
via the semigroup:
$$
e^{-t\Ell}f(x) =\eex [ f\circ X_t]\mbox{  a.e.}
$$
for any $f\in L^2(X)$ and $t\ge 0$.
A regular Dirichlet $\calE$  form is said to be \textit{local} if
$\calE (f,g) = 0$ whenever $f$ and $ g$ have disjoint support. If
even  $\calE (f,g) =0$ whenever $g$ is constant on a neighborhood of
the support of $f$ then $\calE$ is said to be \textit{strongly
local}. Typical examples of operators associated to strongly local
Dirichlet forms are Laplace-Beltrami operatators on manifolds.
Typical example of operators associated to non-local Dirichlet forms
are   Laplacians on graphs and fractional Laplacians which are the generators
of $\alpha$-stable processes.

We will need the restrictions to open sets of operators associated
to a Dirichlet form.  Let $\Ell$ be the selfadjoint operator on $L^2
(X,m)$ associated to the regular Dirichlet form $\calE$ with domain
$\calD$. Let  $U$ be an open subset of $X$. Then, the restriction of
$\calE$ to
$$\overline{ \calD \cap C_c (U)}^{\|\cdot\|_{\calE}}$$
is a regular Dirichlet form on $L^2 (U,m)$ (this is easy to see, Theorem 4.4.2
from \cite{Fukushima-80} contains it).
This form will be denoted by $\calE_U$.  The selfadjoint operator
associated to $\calE_U$ will be denoted by $\Ell_U$.

Using the \emph{hitting time} of $B:=X\setminus U$
$$
\sigma_B:= \inf \set{t\ge 0}{X_t\in B}
$$
we obtain the following instance of the Feynman-Kac formula that
allows for a probabilistic interpretation of the semigroup generated
by $\Ell_U$:
\begin{equation}\label{Feynman-Kac}
e^{-t\Ell_U}f(x) =\eex\left[ f\circ X_t \cdot
\One_{\set{\omega}{\sigma_B(\omega)>t}}\right]\mbox{  a.e.}
\end{equation}

for $f\in L^2(U)$ and $t\ge 0$; see \cite{Fukushima-80}, Section 4.1
and Thm 4.4.2, p. 111

By construction the  operator $\Ell_U$ acts in $L^2 (U,m)$. In order to
compare the operators  $\Ell$ and $\Ell_U$ and their spectral properties
we will have to extend functions of  $\Ell_U$ from operators on
$L^2 (U)$ to operators on  $L^2 (X)$. We will do so by extending $\varphi(\Ell_U)$
by $0$
on $L^2 (X \setminus U, m)$ for $\varphi\in C_0(\RR)$.
This extension by $0$ is well in line with the idea that $\Ell_U$
arises from $\Ell$ by adding a potential with value $\infty$ on
$X\setminus U$ which would mean that $\Ell_U=\infty$ on $L^2 (X \setminus U, m)$.
For our later considerations the extension of $e^{-t\Ell_U}$ has the advantage
that in this way (\ref{Feynman-Kac}) is  valid for all $f\in L^2
(X)$. This extension will be done tacitly in the sequel; however, in
the proof of Theorem \ref{theorem-decomposition-principle} we have
to be a little more careful in one step.

Further, we will need some notions from potential theory of
Dirichlet forms and use the connection to stochastic processes.  The
\textit{capacity} is a set function associated to a Dirichlet form.
It measures the size of sets adapted to the form. It is defined as
follows: For $U\subset X$, $U$ open, we define
$$
\capp(U) := \inf \{ \left(\calE (v,v) + \|v\|^2_{L^2(X,m)}\right)^{1/2}
\with v \in \calD, \One_U \le v \}, $$
with the usual convention that $\inf \emptyset = \infty$.
For arbitrary $A \subset X$, we then set
\[ \capp(A) := \inf \{ \capp(U) \with A \subset U \} \]
(see \cite{FukushimaOT-94}, Section 2.1). For regular forms
the capacity of relatively compact sets can easily be seen to be
finite. The 1-equilibrium potential $e_B$ of a set of finite capacity can be
thought of as
the minimizer in the variational definition of the capacity given above, see
\cite{FukushimaOT-94},
Theorem 2.1.5. It has an intimate relation to the first hitting time of the
Markov process
$(\Omega, (\pex)_{x\in X},(X_t)_{t\ge 0})$ corresponding to $(\calE, \calD)$ viz
\begin{equation}
 e_B(x)=\eex[e^{-\sigma_B}]\label{hitting}
\end{equation}
for $m$-a.e. $x\in X$ by \cite{Fukushima-80}, formula (4.2.9) on p. 99 and Thm
4.3.5, p. 106.

\medskip

The spectrum of a selfadjoint operator $A$ will be denoted by
$\sigma (A)$ i.e.  $\sigma (A)$ is the set of all real numbers
$\lambda$ such that $A - \lambda I$ is not continuously invertible
(where $I$ denotes the identity). The spectrum of $A$ can be
decomposed in two disjoint parts viz the set of eigenvalues of
finite multiplicity and the remaining part. This remaining part is
known as the \textit{essential spectrum} of $A$ and denoted by
$\sigma_{ess} (A)$. Its characteristic feature is that it is stable
under compact perturbations, a celebrated fact going back to Weyl,
\cite{Weyl, Weyl-09}.

\section{A general decomposition
principle}\label{sec-decomposition}

In this section we will present a decomposition principle for regular Dirichlet
forms. This will
give that the essential spectrum of Dirichlet forms  is stable under
suitable restrictions to complements of compact sets. Throughout the
remainder of this section we let   a regular Dirichlet form $\calE$
on $L^2 (X,m)$ with domain $\calD$ and associated selfadjoint
operator $\Ell$ be given. The crucial
assumption will be the following, taken from \cite{LSW}.

\begin{definition}
 We say that $\Ell$ is \emph{spatially locally compact} if
 $\One_E e^{-\Ell}$ is a compact operator for every measurable $E$ with $m(E) <
\infty$.
\end{definition}

Here and in the sequel we identify the function $\One_E$ with the corresponding
bounded multiplication operator on $L^2(X)$.
\begin{remark}\label{rem}(cf. \cite{LSW}, Remark 2.7)
 \begin{itemize}
 \item[(a)] By \cite{LSW}, Thm 1.3, we could also use the resolvent or spectral
projections to  characterize spatial local compactness.

\item[(b)] As follows from  Theorem 1.3 of \cite{LSW}, spatial local
compactness is equivalent to compactness of the map $\One_E : (\calD,
\|\cdot\|_{\calE})  \longrightarrow L^2 (X)$  for all
measurable $E$ in $X$ with $m(E) < \infty$. In particular, spatial
local compactness of $\Ell$ implies spatial local compactness for all
$\Ell{'}$ such that the form domain of $\Ell{'}$ embeds into $\calD$.

\item[(c)]  If  the semigroup $e^{-t \Ell} $ is \textit{ultracontractive},
 i.e.  $e^{-t \Ell} : L^2 \rightarrow L^\infty$ for some $t > 0$,
  then $\Ell$ is spatially locally compact: In fact $\One_E e^{-t \Ell}$ factors
  through $L^\infty$ and the little Grothendieck theorem gives that it is a
Hilbert-Schmidt operator,   in particular compact. See the discussion in
\cite{Stollmann-94a}, or
\cite{DemuthSSV-95} for the case of
  positivity preserving semigroups. This is what makes our results applicable
to large classes of   manifolds, see the discussion in subsection
\ref{subsection-smooth} below for
more details.

  \item[(d)] Note that $e^{-t \Ell}$ maps into
  $L^\infty$ whenever there exists an  $\alpha>0$ with $D(\Ell^\alpha) \subset
  L^\infty$. Therefore,  the spatial local compactness of $\Ell$ for $X$ being
Euclidean space or a manifold can  sometimes easily be checked in terms of
compactness of Sobolev embeddings, i.e. in variants of
  Rellich's theorem \cite{Kato-66}, Theorem V.4.4.

  \item[(e)] The Laplacian on quantum or metric graphs is spatially locally
compact
under quite general
  assumptions, since its domain is continuously embedded in $L^\infty$, see
\cite{LenzSS-08}.

  \item[(f)] For combinatorial graphs, the condition of spatial local
compactness is
trivially satisfied,
  as $\One_E$ has finite rank in this case. Therefore we get a rather easy and
not very subtle criterion
  in that case, see below for a discussion of more general graphs.
 \end{itemize}
\end{remark}

Here comes the main result of this section.

\begin{theorem}[Decomposition principle]\label{theorem-decomposition-principle}
Let $\calE$ be a regular Dirichlet form  on $L^2 (X,m)$ with
associated selfadjoint operator $\Ell$ that is spatially locally compact. Let
$B\subset X$ be a closed
 set of finite capacity and denote $G:=X \setminus B$. Then,  the operator
 $\varphi (\Ell_G) - \varphi (\Ell)$ is compact for every $\varphi \in C_0 (\RR)$. In
particular,
$$\sigma_{ess} (\Ell_G) = \sigma_{ess} (\Ell)$$
holds.
\end{theorem}

\medskip

The remainder of this section is devoted to a proof of this
decomposition principle. In fact, we will start with a result that says that
the
boundary condition that defines $e^{-\Ell_G}$ is not felt too much near infinity.
\begin{proposition}\label{prop-normest}
 Let $\calE$ be a regular Dirichlet form  on $L^2 (X,m)$ with
associated selfadjoint operator $\Ell$. Let $B\subset X$ be a closed
 set of finite capacity and denote $G:=X \setminus B$. Then, for measurable
$A\subset X$ and  $t>0$ we have
 \begin{eqnarray}\label{normest}
\|\One_A \left(e^{-t\Ell} -e^{-t\Ell_G}\right)\| &\le& \sup_{x\in
A}\left[\pex\{\omega\mid \sigma_B(\omega)\le t\}\right]^\frac{1}{2}\\
&\le& e^\frac{t}{2} \sup_{x\in A} \left[ e_B(x)\right]^\frac{1}{2}
 \end{eqnarray}
\end{proposition}
So, the fact that the Dirichlet boundary condition on $B$ is ``not felt
too much'' is measured in terms of the operator norm, while ``near infinity''
means on sets $A$ far away from $B$ in the sense that the probability of hitting
$B$ during the time interval $[0,t]$  when starting in $A$ is small.

This proposition contains one of main new ingredients of our method of proof
even though analogous calculations involving the Feynman--Kac formula and
Cauchy--Schwarz have been around for quite some time. See, e.g., the proof of
Theorem B.1.1 in \cite{Simon}.

\begin{proof}
 The proof uses an argument from \cite{Stollmann-94a} involving the Feynman-Kac
formula, (\ref{Feynman-Kac}) above, by which we get that,
 for any $f\in L^2$ and a.e. $x\in X$:
 \begin{eqnarray*}
  |(e^{-t\Ell} -e^{-t\Ell_G})f(x)|&=& \left| \eex\left[ f\circ X_t\left( 1 -
\One_{\set{\omega}{\sigma_B(\omega)>t}}\right)\right]
 \right|\\
 &=&
 \left| \eex\left[ f\circ X_t\One_{\set{\omega}{\sigma_B(\omega)\le
t}}\right]\right|\\
 &\le&
 \eex\left[|f|^2\circ X_t\right]^\frac12\pex\set{\omega}{\sigma_B(\omega)\le
t}^\frac12\\
 \end{eqnarray*}
by Cauchy--Schwarz. Note that the integral appearing in the  first factor gives
$$
\eex\left[|f|^2\circ X_t\right]=e^{-t\Ell}(|f|^2)(x)
$$
an integrable function, with integral bounded by $\| f\|_2^2$ as the
semigroup of a Dirichlet form is sub-Markovian (and, hence, induces
an operator from $L^1$ to $L^1$ with norm not exceeding $1$). For
the second factor, we have, by \eqref{hitting} above
 \begin{eqnarray*}
  \pex\set{\omega}{\sigma_B(\omega)\le t}&\le&
e^t\eex\left[e^{-\sigma_B}\right]\\
  &=&e^t e_B(x) ,
 \end{eqnarray*}
where $e_B$ is the one-equilibrium potential. Putting these ingredients together
we get

\begin{eqnarray*}
\|\One_A(e^{-t\Ell} -e^{-t\Ell_G})f\|^2&=& \int_A |(e^{-t\Ell}
-e^{-t\Ell_G})f(x)|^2dm(x)\\
&\le& \int_A e^{-t\Ell}(|f|^2)(x)\cdot \pex\set{\omega}{\sigma_B(\omega)\le
t}dm(x)\\
&\le&  \sup_{x\in A}\pex\set{\omega}{\sigma_B(\omega)\le t} \int_A
e^{-t\Ell}(|f|^2)(x)dm(x)\\
&\leq & \sup_{x\in A}\pex\set{\omega}{\sigma_B(\omega)\le t}
\|f\|^2\\
&\leq & e^t \sup_{x\in A}  e_B (x) \;  \|f\|^2,
 \end{eqnarray*}
which gives the assertion. \end{proof}

\begin{corollary}\label{cor-limit}
 Let $\calE$ be a regular Dirichlet form  on $L^2 (X,m)$ with
associated selfadjoint operator $\Ell$. Let $B\subset X$ be a closed
 set of finite capacity and denote $G:=X \setminus B$. Then there is a
 sequence $(M_n)_{n\in\NN}$ of sets of finite measure such that for every $t>0$
 \begin{equation}\label{limit}
 e^{-t\Ell} -e^{-t\Ell_G} = \| \cdot  \| -\lim_{n\to\infty}\One_{M_n}(e^{-t\Ell}
-e^{-t\Ell_G})
 \end{equation}
\end{corollary}

\begin{proof}
We let
$$
M_n:=\set{x\in X}{e_B(x)>\frac{1}{n}}
$$
for a measurable realization of $e_B$. As $e_B$ is square integrable, $M_n$
has finite measure.

Using the preceding Proposition for $A_n:=X\setminus M_n$ we get the
assertion.
\end{proof}

\begin{proof}[Proof of Theorem \ref{theorem-decomposition-principle}]
Let
$A_{comp}$ be the set of those functions $\varphi\in C_0(\RR)$ for which
$\varphi(\Ell)-\varphi(\Ell_G)$ is compact.

\medskip

\rm{1st Step:}  Using the preceding corollary  and the assumption
that $\Ell$ is spatially locally compact, we will see that
$\exp(-\cdot)\in A_{comp}$.

\medskip

In fact, the assumption on $\Ell$ guarantees that
$\One_{M_n}e^{-t\Ell}$ is compact for every $n\in\NN$. Since $\Ell_G\ge \Ell$, Corollary
1.5 from \cite{LSW} gives
that $\One_{M_n}$ is also $\Ell_G$-relatively compact, i.e., $\One_{M_n}e^{-t\Ell_G}$
is
compact for every $n\in\NN$. Hence, $\One_{M_n}(e^{-t\Ell} -e^{-t\Ell_G})$ is compact
and so is its
limit $e^{-t\Ell} -e^{-t\Ell_G}$, by (\ref{limit}) above.

\medskip

\rm{2nd Step:} $A_{comp}=C_0(\RR)$, i.e.  $\varphi (\Ell_G) - \varphi (\Ell)$ is
compact for every $\varphi \in C_0 (\RR)$.

\medskip

We use a Stone-Weierstra{\ss} argument, more precisely \cite{ReedS-72}, Theorem
IV.9 and first record that
we can ignore the left halfline  since both $\Ell$ and $\Ell_G$ are nonnegative
operators.
Clearly,
$A_{comp}$ is closed with respect to uniform convergence. By what we checked in
the Step 1, we
already know that $A_{comp}$  separates the points and that for every $x_0\in
\RR$ there is a
$\varphi\in A_{comp}$ that is nonzero at $x_0$.
It remains to show that $A_{comp}$ is an algebra. Clearly, $A_{comp}$ is a
vector space.
For $\varphi,\psi\in A_{comp}$ we see that
\begin{equation*}
(\varphi\cdot\psi)(\Ell)-(\varphi\cdot\psi)(\Ell_G) = \varphi(\Ell)(\psi(\Ell)-\psi(\Ell_G))+
(\varphi(\Ell)-\varphi(\Ell_G))\psi(\Ell_G)
\end{equation*}
is compact and this finishes the second step.

\medskip

 \rm{3rd Step:} The 2nd Step shows the first claim of the theorem.
 Given this, the 'In particular' statement of the theorem then follows as the
essential spectrum is stable under compact perturbations. This
stability  can essentially be found e.g. in   \cite{ReedS-78},
Theorem XIII.14. However, $\Ell$ and $\Ell_G$ live on different Hilbert
spaces so that the latter result cannot be applied verbatim. At the
risk of appearing pedantic, we include the complete argument:

By \cite{ReedS-78}, Lemma 2 from Section XIII.4, called the strong spectral
mapping theorem, we infer that
\begin{equation*}
\sigma_{ess}(\Ell)=
\set{\frac{1}{\lambda}-1}{\lambda\in\sigma_{ess}((\Ell+1)^{-1})\setminus
\{0\}  },
\end{equation*}
as well as the respective formula for $ \sigma_{ess}(\Ell_G)$.
Weyl's classical result on stability of the essential spectrum under compact
perturbations,
see \cite{Weyl-09} as well as the discussion in \cite{ReedS-78}, Section
XIII.4,
gives that
$$
\sigma_{ess}((\Ell+1)^{-1})=\sigma_{ess}((\Ell_G+1)^{-1}\oplus 0) ,
$$
where the latter denotes the canonical extension of $(\Ell_G+1)^{-1}$ (which is
defined as a bounded
operator on $L^2(G,m)$) to all of $L^2(X,m)$. It follows that
$$
\sigma_{ess}((\Ell+1)^{-1}) \setminus \{0\}
=\sigma_{ess}((\Ell_G+1)^{-1})\setminus \{0\},
$$
and the above strong spectral mapping result then gives the claim.
\end{proof}

\section{A Persson type theorem}\label{sec-Persson}
In this section we present a Persson type theorem for general
regular Dirichlet forms. As discussed in the introduction such a
theorem is a rather direct consequence of a decomposition principle.

\bigskip

We first recall the following general lower bound on the essential
spectrum. The result is   well-known and variants have been used in
several  places, see e.g. \cite{DLMSY,Fujiwara,Sturm-94}. In the
form given here it can be found in the recent work \cite{HKW}.

\begin{lemma}\label{lower-bound} Let $\calE$ be a regular Dirichlet
form on $L^2 (X,m)$. If $(f_n)_{n\in\NN}$ is a sequence in $\calD$ with
$\|f_n\| =1$ for all $n\in\NN$ and $f_n\to 0$ weakly in $L^2 (X,m)$
and $\lambda = \lim_{n\to \infty}\calE (f_n, f_n)$. Then,
$$\inf
\sigma_{ess} (\Ell) \leq \lambda .$$
\end{lemma}

Whenever $X$ is a locally compact Hausdorff space, the family of all
compact subsets of $X$ forms a net. We will need to take limits
along this net. Such limits will be written as $\lim_{K\to X}$.

\begin{theorem}[Persson's theorem]\label{Persson's-theorem}
Let $\calE$ be a regular Dirichlet form  on $L^2 (X,m)$ with domain
$\calD$ and  associated selfadjoint operator $\Ell$ that is spatially locally
compact.  Then,
$$\inf \sigma_{ess} (\Ell) =\lim_{K\to X} \inf \sigma (\Ell_{X \setminus K}).$$
\end{theorem}
\begin{proof}
The infimum of the spectrum is characterized by the variational
principle via
$$\inf \sigma (\Ell_U) = \inf\{\calE_U (f,f) : f\in \calD (\calE_U), \|f\|
=1\}$$ for any open subset $U\subset X$. This easily gives the
monotonicity property
$$\inf \sigma (\Ell_{X\setminus K_1}) \leq \inf \sigma (\Ell_{X\setminus
K_2})$$ whenever $K_1, K_2$ are compact with $K_1 \subset K_2$. This
in turn implies  that the limit in question  exists (with possible
value $\infty$) and actually equals
$$\sup_{K} \inf \sigma (\Ell_{X\setminus K}),
$$ where the supremum is taken  over all compact $K\subset X$. We
will call this limit $\lambda_0$.  We will now show two
inequalities.

\medskip

\textit{ $\inf \sigma_{ess} (\Ell) \leq \lambda_0$.}

\smallskip

This follows  from the previous lemma. Indeed, we can choose a
sequence $(K_n)_{n\in \NN}$  of compact subsets of $X$ with
$\lambda_0 = \lim_{n\to \infty} \inf \sigma (\Ell_{X\setminus K_n})$.
It follows from the discussion  at the beginning of the proof and,
in particular, the monotonicity property that we can assume without
loss of generality $K_n \subset K_{n+1}$ for all $n$ as well as  $X
= \bigcup_n K_n$. By definition of the restriction of $\Ell$ to the
complement of a compact set and the variational principle, we can
then pick a sequence  of functions $f_n \in C_c (X\setminus K_n)
\cap \calD$ with $\|f_n\|_{L^2 (X,m)} =1$ and
$$|\inf \sigma (\Ell_{X\setminus K_n}) - \calE (f_n, f_n)| \leq
\frac{1}{n}
$$
for all $n\in \NN$. Then, the $f_n$ vanish on $K_n$, $n\in\NN$, and
hence converge weakly to $0$ by our assumptions on the $K_n$.
Moreover, by construction $\lambda_0 = \lim_{n\to\infty} \calE
(f_n,f_n)$. Now, the previous lemma gives the desired inequality.

\medskip

\textit{$\inf \sigma_{ess} (\Ell) \geq \lambda_0$.}

\smallskip

The decomposition principle  gives $\inf \sigma_{ess} (\Ell_{X\setminus
K}) = \inf \sigma_{ess} (\Ell_{X})$. Given this, the desired statement
follows from the  (obvious) inequality
$$\inf \sigma (\Ell_{X\setminus
K}) \leq \inf \sigma_{ess} (\Ell_{X\setminus K}).$$

\medskip

Putting together the preceding two inequalities finishes the proof
of the theorem.
\end{proof}

\begin{remark}
\begin{itemize}

\item[(a)] The theorem deals with  regular Dirichlet forms provided
they are spatially locally compact. In particular the theorem covers
Laplace Beltrami operators on manifolds, Laplacians on rather
general  graphs and $\alpha$-stable processes. Details are discussed
in the next section.

\item[(b)] The limit over the compact subsets could be replaced by a
limit over arbitrary measurable relatively compact subsets (as can
be seen from simple monotonicity arguments).

\item[(c)] The limit appearing in the theorem can be expressed in
various ways.  Among them we mention the following equalities:
\begin{eqnarray*}
 &  \lim_{K\to X} & \inf \sigma (\Ell_{X \setminus K})\\
 &=& \sup_{\mbox{$K\subset X$
cpt}}
 \inf \sigma (\Ell_{X \setminus K})\\
&=&\sup_{\mbox{$K\subset X$ cpt}} \inf\{\frac{\calE (f,f)}{\|f\|^2}
: f\in \calD (X \setminus K), f\neq 0\}\\
 &=&\sup_{\mbox{$K\subset X$ cpt}} \inf\{\frac{\calE
(f,f)}{\|f\|^2} : f\in \calD\cap C_c (X \setminus K), f\neq 0\}.
\end{eqnarray*}
Here, the first equality holds by simple monotonicity arguments (as
discussed in the proof of the theorem), the second equality is just
the Rayleigh-Ritz variation principle and the last equality follows
easily from $\calD (X \setminus K) = \overline{\calD\cap C_c (X
\setminus K) }^{\|\cdot\|_{\calE}}$ (which in turn was discussed
above).
\end{itemize}
\end{remark}

\section{The range of applications -- classes of examples }
\label{examples}
We now present three classes of examples satisfying the
considerations of the preceding  sections.

\subsection{Regular Dirichlet forms on manifolds}
\label{subsection-smooth} As mentioned in Remark \ref{rem} (c), the
assumption of spatial local compactness is satisfied provided the
corresponding heat semigroup is ultracontractive. In particular
spatial local compactness holds for the Laplacian on Euclidean
space and for Laplace-Beltrami operators on large classes of
Riemannian manifolds, so that Theorems
\ref{theorem-decomposition-principle} and \ref{Persson's-theorem}
apply.

 In this way we recover the
corresponding
decomposition prinicple  of Donnelly and Li \cite{DL} and
corresponding instances of Persson's theorem of Grillo
\cite{Grillo}. As mentioned above, Grillo's work does not contain a
decomposition principle and the method of proof is limited to local
forms. Although not stated explicitly, a sort of local compactness
is needed, in that results from \cite{BiroliT-97} are used. In that
latter paper, the authors assume volume doubling and a Poincar\'e
inequality. This, in turn, implies ultracontractivity, and, moreover,
very precise pointwise estimates on Greens functions as shown in
\cite{BiroliM-95}. So the compactness properties used in
\cite{Grillo} are much stronger than the one we assume in our main
theorems.

Let us now discuss the case of manifolds a little more detailed and
note first, that we can tacitly assume that we are dealing with a
non-compact Riemannian manifold, as otherwise the appearing
essential spectra are empty, anyway. We will always consider
complete and connected weighted manifolds $(M,\mathbf{g},\mu)$ as in
\cite{Grigoryan-09}; the corresponding Laplacian $\Delta_\mu$ is
defined as in Section 4.2 of the latter reference and its negative,
$- \Delta_\mu$,  generates a regular Dirichlet form, the reader is
referred to \cite{Grigoryan-09} again for a thorough discussion of
this little more general Laplace-Beltrami operator together with the
necessary background from Riemannian geometry.
A detailed account of heat kernel estimates
and the underlying geometric properties can also be found in
\cite{Grigoryan-99}, to which we refer the reader instead of
recording the relevant material here. Let us first note that
ultracontractivity does not hold on any manifold; see e.g., Example
9.17, p. 255 in \cite{Grigoryan-09}. It is true for the following
classes:
\begin{itemize}
 \item \emph{Cartan--Hadamard} manifolds, defined by the property that there
sectional curvature $K(p)$ is non-positive. This is the class of manifolds
considered in the work of Donnelly and Li, \cite{DL}; see \cite{Grigoryan-99},
Section 7.4 for ultracontractivity in this case.
\item \emph{Minimal submanifolds} of $\RR^N$; see \cite{Grigoryan-99},
Section 7.3 for ultracontractivity in this case.
\item \emph{Manifolds of bounded geometry}, defined by a uniform lower bound on
 Ricci curvature together with a positive uniform lower bound on the
injectivity radius; see \cite{Grigoryan-99},
Section 7.5.
\end{itemize}
Summarizing these instances we get
\begin{corollary}
Assume that $(M,\mathbf{g},\mu)$ belongs to the above mentioned
classes and let $\Delta$ be the associated Laplacian. Then
 \begin{enumerate}
  \item The heat semigroup is ultracontractive, in particular $-\Delta$ is
spatially locally compact.
  \item For every compact subset $K\subset M$:
  $$\sigma_{ess}( -\Delta)=\sigma_{ess}( - \Delta_{M\setminus K}).$$
  \item Perssons' theorem holds:
  $$ \inf \sigma_{ess}(- \Delta)=\sup_{K}  \inf \sigma(- \Delta_{M\setminus K}),$$
  where the supremum is taken over the net of compact subsets of $M$.
   \end{enumerate}
\end{corollary}
The case of manifolds with bounded geometry is a generalization of the case of
non-negative Ricci curvature where
heat kernel bounds go back to the celebrated work of Li and Yau, \cite{LY}. Far
reaching generalizations have been obtained
in \cite{Rose-16} where uniform lower bounds are replaced by integral bounds.

Using our results and the fact that sectional curvature controls isoperimetry
and this gives, in turn, lower bounds on
the spectrum of the Laplacian just as in \cite{DL}, we get the  following
generalization of the result of Donnelly and Li from the latter reference:

\begin{corollary}
 Assume that $(M,\mathbf{g},\mu)$ belongs to the above mentioned classes and
that
 sectional curvature satisfies: $K(p)\to - \infty$ for $p\to\infty$,
then
  $$\sigma_{ess}(- \Delta)=\emptyset.$$
\end{corollary}

\subsection{Fractional Laplacians -- $\alpha$-stable processes}
Our result also applies to various classes of jump processes
including  $\alpha$-stable processes. Dirichlet forms of such
processes have attracted tremendous attention in various respects in
recent years, including study of  eigenvalue estimates \cite{FG,FLS}
and  heat kernel estimates and estimates for solutions of the
corresponding partial differential equations
\cite{BBCK,BGK,CK,GHH,GHL}.

Let us first present the basic model. Let $X = \RR^d$ and $\Delta$
be the usual Laplacian on $L^2 (\RR^d)$. The fractional Laplacian
$\Ell_\alpha=(-\Delta)^{\alpha/2}$ is the infinitesimal generator of a
Markov process known as \textit{$\alpha$-stable process}  for $0 <
\alpha < 2$. The corresponding Dirichlet form is (up to a constant)
defined by
\[ \calE_\alpha(u, v) := \int_{X \times X \setminus D} (u(x) - u(y)) (v(x) -
v(y)) |x-y|^{-d-\alpha} \: \rmd x \rmd y, \]
\[ \calD_{\alpha} := \{ u \in L^2(X) \with \calE_{\alpha}(u) < \infty \}, \]
where $D:=\{(x,x) : x\in X\} \subset X\times X$ denotes the
diagonal.

This is obviously a non-local Dirichlet form.  It is regular since
by basic theory of Sobolev spaces the test functions are a core for
this Dirichlet form. Now, clearly a power of $\Ell_\alpha$, viz
$-\Delta$, is spatially locally compact (see Section
\ref{subsection-smooth}) and so is then  $\Ell_\alpha$ itself, by
\cite{LSW}, Thm 1.3. Hence, both the decomposition principle and the
Persson theorem of the previous section apply to $\Ell_\alpha$.

As mentioned in the introduction, this seems to be  the first result
on a decomposition principle for an operator that is neither a
Laplacian on a manifold nor on a graph.

More generally, consider $U \subset \RR^d$ open
$$
\calE(u, v) := \int_{U \times U \setminus D} (u(x) - u(y)) (v(x) -
v(y)) j(x,y) \: \rmd x \rmd y
$$
on
$$
\calD _0:= \{ u \in C_c(U) \with \calE(u) < \infty \} ,
$$
with a symmetric positive weight function $j$. This defines a closable form and
we denote by $\calD$ the domain of its closure and by $\Ell$ the corresponding
non-negative operator.

\begin{corollary}
 Assume that $j, \calE, \Ell$ are as above and that additionally:
\begin{itemize}
\item [{\rm (i)}] $\calD_0$  is dense in $C_c(U)$ with respect to
$\|\cdot\|_\infty$.
\item [{\rm (ii)}] For some $\alpha\in (0,2), c>0$:
  $$
 j(x,y)\ge c|x-y|^{-d-\alpha}\mbox{  for all  }x,y\in\RR^d .
 $$
\end{itemize}
 Then
 \begin{enumerate}
  \item $\Ell$ is spatially locally compact.
  \item For every compact subset $K\subset U$:
  $$\sigma_{ess}(\Ell)=\sigma_{ess}(\Ell_{U\setminus K}).$$
  \item Perssons' theorem holds:
  $$ \inf \sigma_{ess}(\Ell)=\sup_{K} \inf \sigma(\Ell_{U\setminus K}),$$
  where the supremum is taken over the net of compact subsets of $U$.
 \end{enumerate}
\end{corollary}
The proof is evident: by assumption it follows that the form domain
of $\Ell$ embeds into the form domain of $\Ell_\alpha$. Hence,  $\Ell$ is
spatially locally compact by Remark \ref{rem} (b) above so that we
can apply our Theorems \ref{theorem-decomposition-principle} and
\ref{Persson's-theorem}.

Evidently, as well, if $j$ satisfies an upper bound of the form as
in (ii) above, possibly with a different $\alpha'$ then the
denseness assumption (i) is satisfied. Thus our analysis includes
large classes of examples, studied, e.g. in \cite{BBCK,FKV}; see
also the references there for more pointers to the literature. We
should also like to point out that under such stronger conditions,
heat kernel bounds are known which would give an alternative, more
complicated way of showing that the associated operator is spatially
locally compact.

Of course, we can consider the more general case of fractional
Laplacians on manifolds or even jump type Dirichlet forms on metric
measure spaces whenever Gaussian estimates or -  sufficient for our
purpose - just  ultracontractivity estimates are available. See
\cite{BGK,CK-2,GHL,HK} for a start.

\subsection{Regular Dirichlet forms on discrete spaces}
The study of Laplacians on graphs has a long history (see, e.g., the
monographs \cite{Chung,Col} and the references therein). Much
research has been devoted to study graphs with uniformly bounded
vertex degree. In recent years, various issues related to
unboundedness of the vertex degree have been studied. A glimpse of
these developments can be inferred from  the survey articles
\cite{KellerL-10,Kel} and references therein. The mentioned  issues
can be studied in various settings. The most general setting seems
to be the one introduced in \cite{KellerLenz}, which we now recall:

Let $X$ be a countable set. Let $m$ be a measure on $X$ with full
support, i.e., $m$ is a map on $X$ taking values in $(0,\infty)$. A
\emph{symmetric weighted graph over $X$} or \textit{graph} for short
 is a pair $(b,c)$ consisting of
a map $c : X \to [0,\infty)$ and a map $b : X\times X \to
[0,\infty)$ with $b(x,x)=0$ for all $x\in X$ satisfying the
following two properties:
\begin{itemize}
\item[(b1)] $b(x,y)= b(y,x)$ for all $x,y\in X$.
\item[(b2)] $\sum_{y\in X} b(x,y) <\infty$ for all $x\in X$.
\end{itemize}

To $(b,c)$ we associate the form $\calE^{\rm comp}_{b,c}$ defined on
the set $C_c (X)$ of functions on $X$ with finite support by
$$
\calE^{\rm comp}_{b,c} : C_c (X)\times C_c (X) \longrightarrow
[0,\infty)$$
$$\calE^{\rm comp}_{b,c} (u,v) = \frac{1}{2}
\sum_{x,y\in X} b(x,y) (u(x) - u(y))\overline{(v(x) - v(y))} +
\sum_x c(x) u(x)\overline{v(x)}.$$ Observe that the first sum is
convergent by properties (b1) and~(b2); the second sum is finite.
The form $\calE^{\rm comp}_{b,c}$ is closable in $\ell^2(X,m)$; the
closure will be denoted by $\calE_{b,c,m}$. Now, the forms
$\calE_{b,c,m}$ for $(b,c)$ graph over $X$ are exactly the regular
Dirichlet forms on $\ell^2 (X,m)$
(\cite{KellerLenz,FukushimaOT-94}). Clearly, this setting contains
both combinatorial graphs as well as weighted graphs with uniformly
bounded vertex degree. From our main theorems, we then obtain the
following consequence.

\begin{corollary}
 Let $\calE_{b,c,m}$ be as above with the additional assumption that $m$ is
 bounded below by a positive constant. Then the generator $\Ell_{b,c,m}$ is
spatially  locally compact. In particular, the decomposition
principle,  Theorem \ref{theorem-decomposition-principle} and
Persson's theorem,  Theorem \ref{Persson's-theorem} above apply.
\end{corollary}

As for the \textit{proof} we note that  spatial local compactness is
an evident consequence of the fact that $\One_E$ projects onto a
finite dimensional space for any set $E$ of finite measure.

The corollary contains and complements various earlier results: In
\cite{Keller} the  case of $c=0$, $b$ taking values in $\{0,1\}$
only and $m\equiv 1$ is covered (see \cite{Fujiwara} for
corresponding earlier results dealing with the normalized
Laplacian). The reference \cite{KellerL-10} states that a Persson
theorem holds for graphs once a decomposition principle is known and
notes that a decomposition principle holds for locally finite
graphs.

\section{Schr\"odinger type operators -- measure perturbations}
\label{sec-perturbations} We can carry over the ideas presented
above to also  treat Schr\"odinger operators i.e. measure
perturbations of Dirichlet forms by invoking the methods provided in
\cite{StollmannV-96}. Details are discussed next. We note in passing
that quite general potentials have also been discussed in \cite{BC},
where a Persson type theorem is presented, restricted, however to
the case of perturbations of the classical Laplacian on Euclidean
space.

From \cite{FukushimaOT-94}, we  infer that every $u \in \calD$
admits a \emph{quasi-continuous version} $\tilde{u}$, the latter
being unique up to sets of capacity zero. This allows us to consider
measure potentials in the following way: see \cite{Mazya-64},
\cite{Mazya-85} for the special case of the Laplacian and locally
finite measures, \cite{Stollmann-92} and the references in there.

Let $\calM_0 = \{ \mu : \calB \rightarrow [0, \infty] \mid \mu \mbox{ a measure
}\mu \ll \capp \}$, where $\ll$ denotes absolute continuity, i.e. the property
that $\mu(B) = 0$ whenever $B \in \calB$ and $\capp(B) = 0$. For measures in
$\calM_0$ we explicitly allow that the measure takes the value $\infty$. A
particular example is $\infty_B$, defined by
\[ \infty_B(E) = \infty \cdot \capp(B \cap E), \]
with the convention $\infty \cdot 0 = 0$. Note that for $\mu \in \calM_0$ we
have that $\mu[u, v] := \int_X \tilde{u} \tilde{v} \: \rmd \mu$ is well defined
for $u, v \in \calD(\mu)$, where $\calD(\mu) = \{ u \in \calD \mid \tilde{u} \in
L^2(X, \calB, \mu) \}$. It is easy to see that
\[ \calD(\calE + \mu) := \calD \cap \calD(\mu), \;  (\calE + \mu)[u, v] :=
\calE[u,
v] + \mu[u, v] \]
gives a closed form (not necessarily densely defined). One can check that, e.g.,
$\calE + \infty_B = \calE |_{\calD_0(B^c)}$, where $\calD_0(U) = \{u \in \calD
\mid \tilde{u} |_{U^c} = 0 \text{ q.e.} \}$ and that we get
$$
\calE + \infty_B =\calE_U
$$
provided $U:= B^c$ is open; see the discussion in Section 1 above.

If $\mu^- \in \calM_0$ is form small w.r.t. $\calE + \mu^+$, we can furthermore
define $\calE + \mu = \calE + \mu^+ - \mu^-$ by the KLMN-theorem,
\cite{ReedS-75}, Theorem X.17 and denote the associated selfadjoint operator
(which might be selfadjoint in a smaller Hilbert space!) by $\Ell+\mu$.
Note that for $\mu^+ = 0$ this form boundedness implies that
$\mu^-$ is a Radon measure, i.e., finite on all compact sets.

We now discuss, whether the results from Theorems
\ref{theorem-decomposition-principle} and
\ref{Persson's-theorem} remain valid for $\Ell+\mu$. First note that by the
results from \cite{LenzSV-09},
the spatial local compactness property carries over from $\Ell$ to $\Ell+\mu$, since
we can apply \cite{LenzSV-09},
Theorem 1.3, noting that in the notation of the latter paper $Q(\Ell+\mu)=
\calD(\calE + \mu)\subset Q(\Ell)=\calD$
with continuous embedding.

Thus, the extension of our results to Schr\"odinger type operators
$\Ell+\mu$ essentially boils down to checking, whether Proposition
\ref{prop-normest} and Corollary \ref{cor-limit} hold.

Following \cite{StollmannV-96}, we now introduce the appropriate
Kato condition, so let $\hat{\calS}_K$ and  $c_\alpha(\mu)$ be
defined as in the latter paper, with respect to $\Ell$ (called $H$ in
that latter reference).

We first summarize the approximation results from \cite{StollmannV-96}, using,
that due to our separability assumption here, we can work with sequences
instead of nets.

\begin{lemma}
 Let $\calE$ be a regular Dirichlet form  on $L^2 (X,m)$ with domain
$\calD$ and  associated selfadjoint operator $\Ell$, $\mu_+\in \calM_0$ and
$\mu_-\in \hat{\calS}_K$ with
$c_\alpha(\mu_-)<1$ for some $\alpha >0$. Then $\mu_-$ is form small w.r.t. $\Ell$.
\begin{itemize}
 \item [\rm (1)] Then $\mu_-$ is form small w.r.t. $\Ell$.
 \item [\rm (2)] There is a sequence $(V_n)$ in $L^\infty(X,m)_+$ such that
 $$
 \Ell-\mu_-+V_n\to \Ell+\mu\mbox{  as  }n\to\infty
 $$
 in the strong resolvent sense.
 \item [\rm (3)] There is a sequence $(W_n)$ in $L^\infty(X,m)_+$
 such that,
 $$
 c_\alpha(W_n)\le c_\alpha(\mu_-)\mbox{  for all  }n\in\NN .
 $$
 and
 $$
 \Ell-W_n\to \Ell-\mu_-\mbox{  as  }n\to\infty
 $$
 in the strong resolvent sense.
 \item [\rm (4)] For the semigroup differences we get the pointwise estimates
 $$
  |(e^{-t(\Ell+\mu)} -e^{-t(\Ell_G+\mu)})f(x)|\le
(e^{-t(\Ell-\mu_-)}-(e^{-t(\Ell_G-\mu_-)}|f|(x).
  $$
 \item [\rm (5)]  $e^{-t(\Ell+\mu)}$ maps from $L^1$ to $L^1$ provided
$c_\alpha(\mu_-)<1$ for some $\alpha >0$.
\end{itemize}
\end{lemma}
\begin{proof}
 (1) follows from \cite{StollmannV-96}, Thm. 3.1.

 For (2) we can use the arguments from the proof of Theorem 6.1, part (ii) in
the latter reference using, as mentioned above, that an increasing sequence of
compact sets can be used instead of a net.

 (3) follows from \cite{StollmannV-96}, Thm. 3.5, where again, we can choose a
sequence instead of a net.

 (4) follows from (2) and the Feynman--Kac formula, respectively the Trotter
product formula.

 (5) is a consequence of \cite{StollmannV-96}, Thm. 3.3.
 \end{proof}

Note that (4) in the above lemma is reminiscent of results in
\cite{HundertmarkS-04}. We are now in position to generalize
\ref{prop-normest} to Schr\"odinger type operators.

\begin{proposition}\label{prop_schr_normest}
 Let $\calE$, $\Ell$ be as above, $\mu_+\in \calM_0$ and $\mu_-\in \hat{\calS}_K$
with
$c_\alpha(\mu_-)<\frac12$ for some $\alpha >0$. Then there exists $C$,
depending on $\alpha >0,c_\alpha(\mu_-),t>0$ such that the following holds: for
$B\subset X$  a closed
 set of finite capacity and $G:=X \setminus B$, for measurable
$A\subset X$ and  $t>0$ we have
 \begin{eqnarray}\label{schr_normest}
\|\One_A \left(e^{-t(\Ell+\mu)} -e^{-t(\Ell_G+\mu)}\right)\| &\le& C\cdot\sup_{x\in
A}\left[\pex\{\omega\mid \sigma_B(\omega)\le t\}\right]^\frac{1}{2}\\
&\le&C\cdot\sup_{x\in A}e^\frac{t}{2}\left[ e_B(x)\right]^\frac{1}{2}
 \end{eqnarray}
\end{proposition}
\begin{proof}
First of all, Part (4) from the preceding lemma shows that we can
assume that $\mu_+=0$. By Part (3) we can replace the negative part
of $\mu$ by a bounded function that satisfies the same Kato
condition, call it $W$. Therefore, we are left to estimate $\|\One_A
\left(e^{-t(\Ell-W)} -e^{-t(\Ell_G-W)}\right)\|$:
  \\for any $f\in L^2$ and a.e. $x\in X$:
  \begin{small}
\begin{eqnarray*}
  |(e^{-t(\Ell-W)} -e^{-t(\Ell_G-W)})f(x)|&=& \left| \eex\left[ f\circ
X_te^{-\int_0^t W\circ X_sds}\left( 1 -
\One_{\set{\omega}{\sigma_B(\omega)>t}}\right)\right]
 \right|\\ &=&
 \left| \eex\left[ f\circ X_te^{-\int_0^t W\circ
X_sds}\One_{\set{\omega}{\sigma_B(\omega)\le
t}}\right]\right|\\
 &\le&
 \eex\left[|f|^2\circ X_t\cdot e^{-2\int_0^t W\circ
X_sds}\right]^\frac12\pex\set{\omega}{\sigma_B(\omega)\le
t}^\frac12
 \end{eqnarray*}
 by Cauchy--Schwarz. Note that the integral appearing in the  first factor gives
$$
\eex\left[|f|^2\circ X_te^{-2\int_0^t W\circ
X_sds}\right]=e^{-t(\Ell-2W)}(|f|^2)(x)
$$
\end{small}

an integrable function  of $x$, with integral bounded by
$e\|^{-t(\Ell-2W)}\|_{1,1}\| f\|_2^2$.
This is finite as the semigroup maps $L^1$ to itself by \cite{StollmannV-96},
Thm. 3.3 or part (5) of the preceding Lemma, where we use that $c_\alpha(2W) =2
c_\alpha(W)$. Moreover the norm from  $L^1$ to  $L^1$, indicated by the
subscript $1,1$ above can be estimated in terms of the given quantities. This
gives the asserted estimate just like in the proof of Proposition
\ref{prop-normest} above.
\end{proof}

The remaining steps in the proof of our main results can be easily
adapted as well to give:

\begin{theorem}\label{main_schr} Let $\calE$ be a regular Dirichlet form  on
$L^2 (X,m)$ with
associated selfadjoint operator $\Ell$ that is spatially locally compact.  Let
$\mu_+\in \calM_0$ and $\mu_-\in \hat{\calS}_K$
 with $c_\alpha(\mu_-)<\frac12$ for some $\alpha >0$.
 \begin{itemize}
  \item [\rm (1)] Let $B\subset X$ be a closed set of finite capacity and
denote $G:=X \setminus B$. Then,  the operator $\varphi (\Ell_G) - \varphi (\Ell)$ is
compact for every $\varphi \in C_0 (\RR)$. In particular, $$\sigma_{ess}
(\Ell_G+\mu) = \sigma_{ess} (\Ell+\mu)$$
 holds.
  \item [\rm (2)] $$\inf \sigma_{ess} (\Ell+\mu) =\lim_{K\to X} \inf \sigma (\Ell_{X
\setminus K}+\mu).$$
 \end{itemize}
\end{theorem}

\bigskip

\section*{Acknowledgments} D.L.  gratefully acknowledges many
inspiring discussions with Matthias Keller on a wide range of topics
related to the present paper as well as partial support form German
Research Foundation (DFG).

\def\cprime{$'$}

\end{document}